\documentclass{birkjour}
\usepackage{amsmath}
\usepackage{amssymb}
\usepackage{amsthm}

\usepackage{verbatim}

\usepackage[utf8]{inputenc}

\usepackage{graphicx}
\usepackage{color}
\usepackage{enumerate}

\usepackage{esint}
\usepackage[all]{xy}
\usepackage{framed}
\usepackage{bbm}
\usepackage{subcaption}

\theoremstyle{plain}

\newtheorem{proposition}{Proposition}[section]

\theoremstyle{definition}

\newtheorem{remark}{Remark}[section]

\numberwithin{equation}{section}

\newcommand{\bfe}{\mathbf{e}}

\newcommand{\Sn}{\mathcal{S}^n}

\newcommand{\cH}{\mathcal{H}}

\newcommand{\mR}{\mathbb{R}}
\newcommand{\oO}{\overline{\Omega}}

\DeclareMathOperator{\diag}{diag}

\begin{document}

\author[K. K. Brustad]{Karl K. Brustad}
\title[Special Lagrangian potential equation]{Counterexamples to the\\ comparison principle in the special Lagrangian potential equation}
\address{Frostavegen 1691\\ 7633 Frosta\\ Norway}
\email{brustadkarl@gmail.com}

\begin{abstract}
For each $k = 0,\dots,n$ we construct a continuous \emph{phase} $f_k$, with $f_k(0) = (n-2k)\frac{\pi}{2}$, and viscosity sub- and supersolutions $v_k$, $u_k$, of the elliptic PDE $\sum_{i=1}^n \arctan(\lambda_i(\cH w)) = f_k(x)$ such that $v_k-u_k$ has an isolated maximum at the origin.

It has been an open question whether the comparison principle would hold in 
this second order equation
for arbitrary continuous phases $f\colon \mR^n\supseteq\Omega\to (-n\pi/2,n\pi/2)$.
Our examples show it does not.

\end{abstract}

\maketitle

\section{Introduction}

The \emph{special Lagrangian potential operator} is the mapping $F\colon \Sn\to\mR$ given by
\[F(X) := \sum_{i=1}^n \arctan(\lambda_i(X))\]
where $\lambda_1(X)\leq\cdots\leq \lambda_n(X)$ are the eigenvalues of the symmetric $n\times n$ matrix $X$.
The corresponding equation
\begin{equation}\label{eq:SL_equation}
F(\cH w) = f(x)
\end{equation}
in $\Omega\subseteq\mR^n$, including the autonomous version
\begin{equation}\label{eq:SL_equation_autonom}
F(\cH w) = \theta,
\end{equation}
has attained much interest since it was introduced in \cite{MR666108}.
For a right-hand side constant $\theta \in (-n\pi/2,n\pi/2)$ the solutions of \eqref{eq:SL_equation_autonom} have a nice geometrical interpretation. The graph of the gradient $\nabla w$ in $\Omega\times \mR^n$ is a \emph{special Lagrangian manifold}. i.e., it is a Lagrangian manifold of minimal area. 
See \cite{MR4179860} and the references therein.

Recently, \cite{MR4147574} were able to prove the comparison principle for \eqref{eq:SL_equation} when $f$ is continuous and avoids the \emph{special phase values}
\[\theta_k := (n-2k)\frac{\pi}{2},\qquad k = 1,\dots,n-1.\]
In Section 3 we show that their proof can be somewhat simplified by applying a result in \cite{brustad2020comparison} valid for equations on the generic form \eqref{eq:SL_equation}. However, our main purpose of this short note is to demonstrate how the comparison principle may fail when $\theta_k\in f(\Omega)$. Interestingly, the comparison principle \emph{is} valid in \eqref{eq:SL_equation_autonom} for \emph{all} $\theta\in\mR$. This follows immediately from the facts that $X\leq Y$ implies $\lambda_i(X)\leq\lambda_i(Y)$, $\lambda_i(X+\tau I) = \lambda_i(X) + \tau$, and that $\arctan$ is strictly increasing. Thus, $F$ is elliptic and $F(X+\tau I) > F(X)$ for all $X\in \Sn$, $\tau>0$. See for example Proposition 2.6 in \cite{brustad2020comparison}.

In our construction of the counterexamples, we shall take advantage of a couple of symmetries in $F$. Firstly, as $F(X)$ only depends on the eigenvalues of $X$, we have $F(QXQ^\top) = F(X)$ for every orthogonal matrix $Q$.  Secondly, since $\arctan$ is odd -- and since $\lambda_i(-X) = -\lambda_{n-i+1}(X)$ and the different eigenvalues are treated equally by $F$ -- it follows that $F$ is odd as well. Moreover, we shall make use of the $n\times n$ \emph{exchange matrix}
\[J := \begin{bmatrix}
0 &  & 1\\
 & \reflectbox{$\ddots$} & \\
1 &  & 0
\end{bmatrix}.\]
It is the corresponding matrix to the reverse order permutation on the set $\{1,2,\dots,n\}$. Obviously, $J$ is symmetric and orthogonal.

\section{The counterexamples}

For $k = 0,\dots,n$, define $v_k, u_k\in C(\mR^n)$ as
\begin{align*}
v_k(x) &:= \frac{1}{4} - \sum_{i=1}^{k}|x_i| + \sum_{i=k+1}^{n}\frac{1}{2}|x_i|^{3/2},\\
u_k(x) &:= -v_{n-k}(Jx),
\end{align*}
and let $f_k\in C(\mR^n)$ be the continuous extension of
\[f_k(x) := -\sum_{i=1}^{k}\arctan\left(\frac{3}{8}|x_i|^{-1/2}\right) + \sum_{i= k+1}^{n}\arctan\left(\frac{3}{8}|x_i|^{-1/2}\right),\qquad x_i\neq 0.\]
Observe that
\[f_k(0) = -k\frac{\pi}{2} + (n - k)\frac{\pi}{2} = \theta_k.\]

\begin{proposition}
The functions $v_k$ and $u_k$ are, respectively, viscosity sub- and supersolutions of the equation
\begin{equation}\label{eq:k_equation}
\sum_{i=1}^{n}\arctan\left(\lambda_i(\cH w)\right) = f_k(x)\qquad\text{in $\mR^n$.}
\end{equation}
\end{proposition}

\begin{proof}
Away from the axes, $v_k$ is smooth with Hessian matrix
\[\cH v_k(x) = \frac{3}{8}\diag(0\dots,0,|x_{k+1}|^{-1/2},\dots,|x_{n}|^{-1/2})\]
and, clearly, $F(\cH v_k(x)) = 0 + \sum_{i=k+1}^{n}\arctan\left(\frac{3}{8}|x_i|^{-1/2}\right) \geq f_k(x)$.

Let $\phi$ be a $C^2$ test function touching $v_k$ from above at a point $x^*\in\mR^n$. We may assume that $x^*_i \neq 0$ for all $i>k$ since no touching is possible otherwise. Thus $\bfe_i^\top\cH\phi(x^*)\bfe_i\geq \frac{3}{8}|x^*_{i}|^{-1/2}$, $i = k+1,\dots n$, and the top $n-k$ eigenvalues of $\cH\phi(x^*)$ are larger than $\frac{3}{8}|x^*_{i}|^{-1/2}$ (respectively, in some order). Likewise, for each $i\leq k$ with $x^*_i \neq  0$ we have $\bfe_i^\top\cH\phi(x^*)\bfe_i\geq 0$ and an additional eigenvalue of $\cH\phi(x^*)$ is non-negative. The remaining second order directional derivatives of $\phi$ at $x^*$ may be arbitrarily negative, providing no bound on the smallest eigenvalues, but that does not matter since
\begin{align*}
F(\cH\phi(x^*))
	&= \sum_{i=1}^{n}\arctan\left(\lambda_i(\cH\phi(x^*))\right)\\
	&\geq \sum_{\substack{i\leq k\\ x^*_i= 0}}-\frac{\pi}{2}
	+ \sum_{\substack{i\leq k\\ x^*_i\neq 0}} 0  + \sum_{i>k}c(x^*_{i}),\qquad c(t) := \arctan\left(\tfrac{3}{8}|t|^{-1/2}\right),\\
	&\geq \sum_{\substack{i\leq k\\ x^*_i= 0}}-\frac{\pi}{2}
	- \sum_{\substack{i\leq k\\ x^*_i\neq 0}}c(x^*_{i})  + \sum_{i>k}c(x^*_{i})\\
	&= f_k(x^*).
\end{align*}
This shows that $v_k$ is a subsolution of \eqref{eq:k_equation} in $\mR^n$ for all $k = 0,\dots,n$.

In order to prove that $u_k(x) = -v_{n-k}(Jx)$ is a supersolution, we first note that
\begin{align*}
- f_{n-k}(Jx)
	&= \sum_{i=1}^{n-k}c((Jx)_i) - \sum_{i=n-k+1}^{n}c((Jx)_i)\\
	&= -\sum_{i=n-k+1}^{n}c(x_{n+1-i}) + \sum_{i=1}^{n-k}c(x_{n+1-i})\\
	&= -\sum_{j=1}^{k}c(x_{j}) + \sum_{j=k+1}^{n}c(x_j), \qquad j := n+1-i,\\
	&= f_k(x).
\end{align*}
Now, let $\psi$ be a test function touching $u_k$ from below at $x^*\in\mR^n$.
Then $\phi(x) := - \psi(Jx)$ touches $v_{n-k}$ from above at $Jx^*$ because $\phi(Jx) = - \psi(x) \geq -u_k(x) = v_{n-k}(Jx)$ for $Jx$ close to $Jx^*$, and $\phi(Jx^*) = - \psi(x^*) = -u_k(x^*) = v_{n-k}(Jx^*)$. Thus, $F(\cH\phi(Jx^*)) \geq f_{n-k}(Jx^*)$. Therefore, since $F$ is odd, rotationally invariant, and $\cH\psi(x) = -J\cH\phi(Jx)J$,
\[F(\cH\psi(x^*)) = - F(\cH\phi(Jx^*)) \leq - f_{n-k}(Jx^*) = f_k(x^*)\]
as claimed.
\end{proof}

We obviously have $v_k(0)-u_k(0) = 1/2>0$. In order to create a counterexample to the comparison principle it only remains to observe that
\begin{align*}
-u_k(x) &= v_{n-k}(Jx)\\
&= \frac{1}{4} - \sum_{i=1}^{n-k}|x_{n+1-i}| + \sum_{i=n-k+1}^{n}\frac{1}{2}|x_{n+1-i}|^{3/2}\\
&= \frac{1}{4} - \sum_{j=k+1}^{n}|x_{j}| + \sum_{j=1}^{k}\frac{1}{2}|x_{j}|^{3/2}
\end{align*}
and
\begin{align*}
v_k(x)-u_k(x) &= \frac{1}{2} + \sum_{i=1}^{n}\frac{1}{2}|x_{i}|^{3/2} - |x_{i}|\\
&= \frac{1}{2} + \sum_{i=1}^{n}\frac{1}{2}|x_{i}|\left(|x_{i}|^{1/2} - 2\right)
\end{align*}
independently of $k$. If we take the domain to be the unit ball in the infinity-norm,
\[\Omega := \left\{x\in\mR^n\;\colon\; |x_i|<1\right\},\]
it follows that $v_k(x)-u_k(x) \leq 0$ whenever $|x|_\infty\leq 1$ and when there is at least one index $j$ such that $|x_j| = 1$. That is, for $x\in\partial\Omega$. 

\begin{remark}
There is nothing special about the exponent 3/2 in the sub- and supersolutions. If we adjust the phase accordingly, any number strictly between 1 and 2 would do.
\end{remark}

\begin{remark}
The subsolutions $v_k$ are not supersolutions and the supersolutions $u_k$ are not subsolutions. Thus, the question of uniqueness of \emph{solutions} in the Dirichlet problem is still open.
\end{remark}

The ideas behind the above constructions are all contained, and therefore best illustrated, by the case $n=2$, $k=1$. Then also $n-k=1$ and, dropping the subscript 1 yields,
\begin{align*}
v(x,y) &= \frac{1}{4} - |x| + \frac{1}{2}|y|^{3/2},\\
u(x,y) &= -v(y,x)
	= -\frac{1}{4} - \frac{1}{2}|x|^{3/2} + |y|,
\end{align*}
with phase
\begin{equation*}
f(x,y) =
\begin{cases}
-\arctan(\frac{3}{8}|x|^{-1/2}) + \arctan(\frac{3}{8}|y|^{-1/2}),\quad & x\neq 0, y\neq 0,\\
-\pi/2 + \arctan(\frac{3}{8}|y|^{-1/2}), & x = 0, y\neq 0,\\
-\arctan(\frac{3}{8}|x|^{-1/2}) + \pi/2, & x\neq 0, y = 0,\\
0, & x = 0, y = 0.
\end{cases}
\end{equation*}
In addition to the difference
\[v(x,y) - u(x,y) = \frac{1}{2} + \frac{1}{2}|x|\left(|x|^{1/2} - 2\right) + \frac{1}{2}|y|\left(|y|^{1/2} - 2\right)\]
the graph of these functions 
over the square
\begin{equation*}
\Omega := \left\{(x,y)\in\mR^2\;:\; |x|<1, |y|<1\right\}
\end{equation*}
are shown in Figure \ref{fig:1}.
\begin{figure}[h]
	\centering\vspace{-12pt}
	\begin{subfigure}[b]{0.45\textwidth}
		\includegraphics{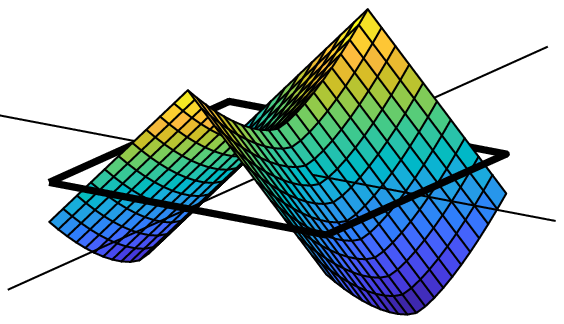}
		\caption{$v(x,y) = \frac{1}{4} - |x| + \frac{1}{2}|y|^{3/2}$.}
	\end{subfigure}
	\begin{subfigure}[b]{0.45\textwidth}
		\includegraphics{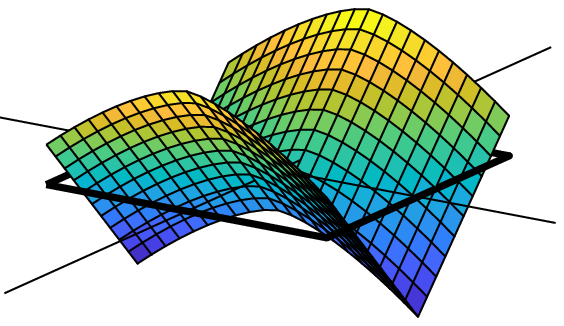}
		\caption{$u(x,y) = - v(y,x)$.}
	\end{subfigure}
	\begin{subfigure}[b]{0.45\textwidth}
		\includegraphics{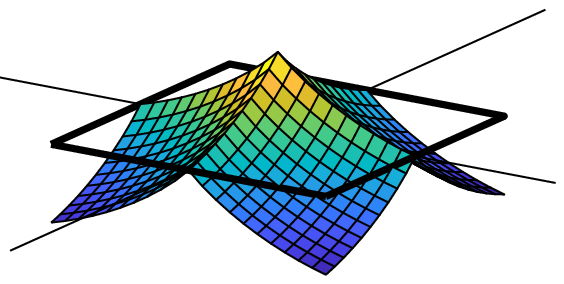}
		\caption{The difference $v-u$ has a strict interior maximum in $\oO$.}
	\end{subfigure}
	\begin{subfigure}[b]{0.45\textwidth}
		\includegraphics{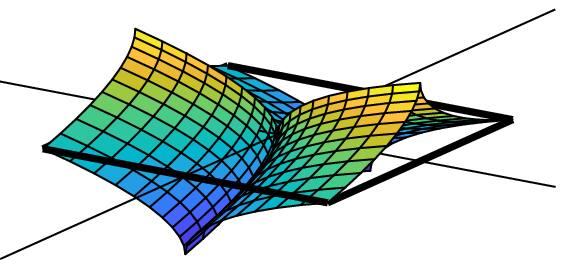}
		\caption{The phase $f$ is continuous, but not differentiable on the axes.}
	\end{subfigure}
	\caption{The case $n=2$, $k=1$ in the square $\Omega$.}\label{fig:1}
\end{figure}

\section{An alternative proof of the comparison principle when the phase does not attain the special values}

Theorem 6.18 \cite{MR4147574} establish
the comparison principle for the equation
\[\sum_{i=1}^n \arctan(\lambda_i(\cH w)) = f(x)\]
in every open and bounded $\Omega\subseteq\mR$ whenever $f\in C(\Omega)$, $\theta_n < f < \theta_0$, and
\[\theta_k := (n-2k)\frac{\pi}{2}\notin f(\Omega),\qquad k = 1,\dots,n-1.\]
The main idea in our proof, as conducted in Example 2.2 in \cite{brustad2020comparison}, is the same as in \cite{MR4147574}. Namely, to reach a contradiction when, for each $i = 1,\dots,n$, $|\lambda_i(X_j)|\to\infty$ as $j\to\infty$ for some sequence $X_j\in\Sn$. Our contribution is to show how this follows almost immediately from a general result for equations on the form $F(\cH w) = f(x)$. For convenience, we reproduce the proof below.

Note that the pathological situation when $f$ takes values outside the interval $[\theta_n,\theta_0]$ comes for free: If, say $f>\theta_0$ somewhere in $\Omega$, the comparison principle vacuously holds since the equation will have no subsolutions. On the other hand, the case $\theta_0\in f(\Omega)$, which is covered by our counterexample, is not pathological. For example, one can easily confirm that $w(x) = |x|^{3/2}$ is a viscosity \emph{solution} to the equation in $\mR^n$ when the right-hand side is the continuous extension of $f(x) := F(\cH w(x))$.


Assume that
\begin{equation}\label{eq:special_phase_value_condition}
\theta_k \notin f(\Omega) \subseteq [\theta_n,\theta_0]
\end{equation}
for all $k = 0,\dots,n$. Proposition 2.7 in \cite{brustad2020comparison} states that the comparison principle will hold if whenever $X_j\in\Sn$ is a sequence such that $\lim_{j\to\infty}F(X_j) = \theta\in f(\Omega)$, then
\[\liminf_{j\to\infty}F(X_j + \tau I) > \theta\]
for every $\tau>0$.
Suppose to the contrary that this is not true. Then there are numbers $\theta\in f(\Omega)$ and $\tau>0$, and a sequence $X_j\in \Sn$ with $F(X_j)\to \theta$, such that
\begin{align*}
0 &= \lim_{j\to\infty} F(X_j + \tau I) - F(X_j)\\
  &= \lim_{j\to\infty} \sum_{i=1}^n \arctan(\lambda_i(X_j) + \tau) - \arctan(\lambda_i(X_j))\geq 0,
\end{align*}
which -- since $\arctan$ is strictly increasing -- is possible only if each $\lambda_i(X_j)$ is unbounded as $j\to\infty$. There is thus a subsequence (still indexed by $j$) such that either $\lambda_i(X_j)\to +\infty$ or $\lambda_i(X_j)\to -\infty$. But this is a contradiction of \eqref{eq:special_phase_value_condition} as
\[f(\Omega)\ni\theta = \lim_{j\to\infty} F(X_j)\\
= \lim_{j\to\infty} \sum_{i=1}^n  \arctan(\lambda_i(X_j))\\
= \sum_{i=1}^n \pm\frac{\pi}{2}
= \theta_k
\]
for some $k=0,\dots,n$.


\bibliographystyle{alpha}
\bibliography{/Users/karlkb/Documents/references.bib}


\end{document}